\documentclass{amsart}
\usepackage{amsmath,amssymb,amsthm,url}
\newtheorem{thr}{Theorem}[section]
\newtheorem{lem}[thr]{Lemma}

\newtheorem{observation}[thr]{Observation}

\theoremstyle{definition}
\newtheorem{defn}[thr]{Definition}
\newtheorem{ex}[thr]{Example}

\newtheorem{pr}[thr]{Problem}

\theoremstyle{remark}
\newtheorem{remr}[thr]{Remark}

\numberwithin{equation}{section}

\def\psd{\operatorname{PSD\, rank}}
\def\R{\mathbb{R}}
\def\ca{\circledast}
\def\tr{\operatorname{tr}}
\def\rk{\operatorname{rank}}
\def\M{\mathcal{M}}
\def\nov{\overline{n}}

\def\H{\mathcal{H}}
\def\A{\mathcal{A}}
\def\B{\mathcal{B}}
\def\C{\mathcal{C}}

\begin{document}

\title[The complexity of positive semidefinite matrix factorization]{The complexity of positive \\ semidefinite matrix factorization}

\author{Yaroslav Shitov}
\address{National Research University Higher School of Economics, 20 Myasnitskaya Ulitsa, Moscow 101000, Russia}
\email{yaroslav-shitov@yandex.ru}






\begin{abstract}
Let $A$ be a matrix with nonnegative real entries. The PSD rank of $A$ is the smallest integer $k$ for which there exist $k\times k$ real PSD matrices $B_1,\ldots,B_m$, $C_1,\ldots,C_n$ satisfying $A(i|j)=\operatorname{tr}(B_iC_j)$ for all $i,j$.
This paper determines the computational complexity status of the PSD rank. Namely, we show that the problem of computing this function is polynomial-time equivalent to the existential theory of the reals.
\end{abstract}

\maketitle

\section{Introduction}

Let $A$ be a nonnegative matrix, that is, a matrix with real nonnegative entries. The \textit{positive semidefinite (PSD) rank} of $A$ is the smallest integer $k$ for which there exist $k\times k$ real PSD matrices $B_1,\ldots,B_m$, $C_1,\ldots,C_n$ satisfying $A(i|j)=\operatorname{tr}(B_iC_j)$ for all $i,j$. (Here and in what follows, $A(i|j)$ stands for an $(i,j)$th entry of $A$.) As explained in~\cite{wcrpsd}, the motivation for the definition of PSD rank came from geometric problems concerning the representation of convex sets for linear optimization. Namely, the PSD rank is a lower bound on the sizes of semidefinite programs that formulate a linear optimization problem over the polytope corresponding to a given matrix, see the discussion in~\cite[p. 10]{PSD}. A strongly related quantity is the \textit{nonnegative rank}, and it is defined in the same way but with an additional assumption that the matrices $B_i$ and $C_j$ are diagonal. Actually, the PSD rank has been introduced as a semidefinite analogue of the nonnegative rank, which serves as a lower bound for the sizes of linear programs~\cite{Yan}. We note that the semidefinite and linear programming relaxations are concepts that belong to a rapidly developing area of modern computer science. We mention the paper~\cite{FMPTdW}, which provides an exponential lower bound for sizes of linear relaxations of several classical NP-complete problems, and the paper~\cite{LRS}, which solves the same problem for semidefinite relaxations.

In our paper, we discuss the computational complexity of evaluating the PSD rank. The decision version of this problem can be formulated as follows.



\begin{pr}\label{prob11} (PSD RANK.) Given $A\in\mathbb{Z}^{m\times n}$ and $r\in\mathbb{Z}$. Is $\operatorname{PSD\, rank}(A)\leqslant r$?
\end{pr}

The complexity status of this problem has been wide open. Actually, it was conjectured that PSD rank is NP-hard to compute, see Problem~9.6 in~\cite{PSD} and a discussion in~\cite[p. 10]{wcrpsd}. In this paper, we deduce this conjecture as a corollary of a stronger result. Namely, we prove that PSD RANK is polynomial-time equivalent to the problem known as the \textit{Existential theory of the reals}. A similar equivalence is still an open question for the nonnegative rank, which is only known to be NP-hard, see the original paper by Vavasis~\cite{Vavasis} and a short proof in~\cite{myshort}.

\section{$\exists\mathbb{R}$}

In this section, we recall the known facts on the existential theory of the reals and formulate our main result. To begin with, we recall the standard description of the yes-instances in Problem~\ref{prob11} as a semialgebraic set. If $B_i$ and $C_j$ are $k\times k$ PSD matrices, we can write their Cholesky decompositions $B_i=\sum_{t=1}^ku_{it}u_{it}^\top$ and $C_j=\sum_{\tau=1}^kv_{j\tau}v_{j\tau}^\top$ with $u_{it},v_{j\tau}\in\mathbb{R}^k$. We get $$\operatorname{tr}(B_iC_j)=\sum_{t,\tau}\operatorname{tr}\left(u_{it}u_{it}^\top v_{j\tau}v_{j\tau}^\top\right)=\sum_{t,\tau}(u_{it}^\top v_{j\tau})^2.$$ Therefore, a matrix $A$ has PSD rank at most $k$ if and only if the equations
\begin{equation}\label{eq11}A(i|j)=\sum_{t,\tau}(u_{it1}v_{j\tau1}+\ldots+u_{itk}v_{j\tau k})^2\end{equation}
have a simultaneous solution $u_{its},v_{j\tau\sigma}\in\mathbb{R}$. In other words, PSD RANK is the problem to decide whether the semialgebraic set~\eqref{eq11} is non-empty. The existential theory of the reals (ETR) is a more general problem in which we allow arbitrary first-order quantifier free formulas over $(\R,+,-,*,0,1)$ instead of~\eqref{eq11}.

\begin{pr}\label{prob12}
(EXISTENTIAL THEORY OF THE REALS.)

\noindent Is a given first-order quantifier-free formula over $\R$ satisfiable?
\end{pr}

Defining $\exists\mathbb{R}$ as the class of decision problems that admit a polynomial reduction from ETR, we see that PSD RANK lies in $\exists\R$.

\begin{remr}\label{remr11}
If $k>\min\{m,n\}$ in Problem~\ref{prob11}, then $(A,k)$ is trivially a yes-instance, so we can assume without loss of generality that $k\leqslant\min\{m,n\}$. Now the length of~\eqref{eq11} is polynomial in the length of $(A,k)$, so~\eqref{eq11} is indeed a polynomial reduction from PSD RANK to ETR.
\end{remr}

The hardest problems in $\exists\R$ are called $\exists\mathbb{R}$-\textit{complete} problems. That is, a problem $\Pi$ is $\exists\mathbb{R}$-complete if and only if there are polynomial reductions from $\Pi$ to ETR and from ETR to $\Pi$. Our main result states that PSD RANK belongs to this class.

\begin{thr}\label{thr11}
PSD RANK is $\exists\mathbb{R}$-complete.
\end{thr}

The proof of Theorem~\ref{thr11} goes as follows. In Section~3, we discuss various problems in $\exists\R$, and we get a suitable restricted version of ETR that is still $\exists\mathbb{R}$-complete. In Section~4, we illustrate the definition of PSD rank on concrete examples, and we mention several known results that we need in our proof. In Section~5, we introduce the matrix completion problem, which is an essential technique of our proof. In Section~6, we study the connection between this problem and PSD RANK, and we get ready to prove the main result. In Section~7, we finalize the proof of Theorem~\ref{thr11} and discuss a related question arisen from a recent paper by Fawzi, Gouveia, and Robinson.

\section{Several $\exists\R$-complete problems}

The ETR problem remains $\exists\R$-complete when restricted to formulas consisting of a single polynomial equation, see Proposition~3.2 in the survey~\cite{Mato}. We need a slightly refined version of this statement, so we reproduce the proof as in~\cite{Mato} for the sake of completeness. We define a \textit{monomial} as $\pm x_{i_1}x_{i_2}\ldots x_{i_n}$, where each $x_{i_k}$ is a variable. We say that a polynomial $f\in\mathbb{Z}[x_1,\ldots,x_n]$ is in the \textit{standard form} if it is written as a sum of monomials.

\begin{thr}\label{lem12}
It is $\exists\mathbb{R}$-complete to decide whether an equation $f(x_1,\ldots,x_n)=0$ has a solution, even if $f$ is written as a sum of monomials.
\end{thr}

\begin{proof}
Let $\Phi$ be a quantifier-free formula. If $\Phi$ is a conjunction of polynomial equations $\Gamma=\{f_1=0,\ldots,f_k=0\}$ in which every $f_i$ is in standard form, then the equation $f_1^2+\ldots+f_k^2=0$ gives the desired reduction after getting rid of brackets by distributivity.

Before we proceed, we note that the use of negations allows us to assume that every atom in $\Phi$ has the form $g>0$ for some expression $g$. We proceed with a construction of the set $\Gamma$ of polynomials which have a simultaneous solution if and only if $\Phi$ is satisfiable. We start with $\Gamma=\varnothing$. For any atom $g>0$, we add to $\Gamma$ the equation
$$\left((gu_g^2-1)^2+(w_g-1)^2\right)\left((g+v_g^2)^2+w_g^2\right)=0,$$
where $u_g,v_g,w_g$ are new variables, and $w_g$ represents the logical value of the corresponding atom (that is, we have $w_g=1$ if $g>0$ and $w_g=0$ otherwise). Further, if the variables $w_a,w_b$ represent the values of Boolean variables $a,b$, then $1-w_a$, $w_aw_b$, $w_a+w_b-w_aw_b$ represent $\neg a$, $a\wedge b$, $a\vee b$, respectively. We can add the corresponding polynomials to $\Gamma$, and eventually we get a variable that represents the value of $\Phi$.

Finally, we assume that an equation $f=0$ in $\Gamma$ is such that $f=(u\star v)\ast w$ with $\ast,\star\in\{+,-,\cdot\}$, and at least one of the expressions $u,v,w$ is not a variable. Then we can add new variables $\varphi_u, \varphi_v, \varphi_w$ and replace $f=0$ in $\Gamma$ by the four equations $u-\varphi_u=0$, $v-\varphi_v=0$, $w-\varphi_w=0$, $(\varphi_u\star \varphi_v)\ast \varphi_w=0$. This way of reasoning eventually allows us to transform all the equations in $\Gamma$ to the form $(y_{i_1}\star y_{i_2})\ast y_{i_3}=0$, where each $y_{j}$ is either a variable or a constant $0$ or $1$. All the discussed transformations require time linear in the length of $\Phi$, so the proof is complete.
\end{proof}

We need a geometric result by Grigoriev and Vorobjov to proceed with a more restricted version of ETR. The following theorem is a special case of Lemma~9 in~\cite{GriVor}.

\begin{thr}\label{thrGV}
Let $L$ be the bit length of a polynomial $f\in\mathbb{Z}[x_1,\ldots,x_n]$. Then every connected component of the set $\{f=0\}$ has a non-empty intersection with the ball of the radius $2^{2^{CL}}$ centered at the origin, where $C$ is an absolute constant.
\end{thr}

Now we prove the $\exists\R$-completeness of the problem used in our reduction.

\begin{thr}\label{lem14}
Theorem~\ref{lem12} remains true under an additional restriction that $f$ is forbidden to have a solution outside the cube $[-1,1]^n$.
\end{thr}

\begin{proof}
Let $f$ be a polynomial as in Theorem~\ref{lem12}. We will construct a polynomial $\varphi$ without solutions outside $[-1,1]^n$ such that the solubility of $\varphi=0$ is equivalent to the solubility of $f=0$. Our construction requires new variables which we denote by $y_0,\ldots,y_m,z_1,\ldots,z_n$. (Here, $m$ is the number $[CL]+1$, where $L$ and $C$ are as in Theorem~\ref{thrGV}.) We define the \textit{homogenization} of $f$ as $h=y_m^d\cdot f(x_1/y_m,\ldots,x_n/y_m)$, where $d$ is the degree of $f$, and we set
$$\varphi=\sum_{j=0}^{m-1}\left(y_{j+1}-y_{j}^2\right)^2+(2y_0-1)^2+\sum_{i=1}^n\left(x_i^2+z_i^2-1\right)^2+h^2.$$
Getting rid of brackets in $\varphi$, we get a polynomial that still satisfies the requirement of Theorem~\ref{lem12}. Note that we have constructed $\varphi$ in polynomial time. 

If the equality $\varphi=0$ is satisfied, the first two summands of $\varphi$ tell us that $y_j=2^{-2^j}$. The third summand shows that $x_i,z_i\in[-1,1]$, so that $\varphi$ satisfies the additional requirement of the present theorem. The fourth summand shows that $h=0$, which implies that the equality $f=0$ is satisfiable since $y_m\neq0$.

Now we check that the satisfiability of $f=0$ implies the satisfiability of $\varphi=0$. If $f(\xi_1,\ldots,\xi_n)=0$, then we can assume without loss of generality that $|\xi_i|<2^{2^m}$ by Theorem~\ref{thrGV}. We construct a point on $\{\varphi=0\}$ as
$$x_i=2^{-2^{m}}\xi_i,\,\,\,y_j=2^{-2^{j}},\,\,\,z_i=\sqrt{1-x_i^2}.$$

Therefore, we have proved that $f\to\varphi$ is a polynomial reduction of the problem in Theorem~\ref{lem12} to the desired restricted version.
\end{proof}

\section{Auxiliary results on psd rank}

We recall several known results on the PSD ranks of the unit matrix, of the entrywise square $A\circ A$ of a matrix, of the sum of matrices, and of block matrices. Example~\ref{ex21} below appeared as Example~2.11 in~\cite{PSD}, and Lemmas~\ref{lem25} and~\ref{lem24} were a part of Theorem~2.9 in the same paper. Lemma~\ref{lemblock} was a part of Theorem~2.10 in~\cite{PSD}, and Observation~\ref{obs123} is an easy consequence of Lemma~\ref{lem25}. We assume that the matrices appearing in these statements are nonnegative and allow the corresponding matrix operations and constructions of block matrices.

\begin{ex}\label{ex21}
$\psd(I_n)=n$.
\end{ex}

\begin{lem}\label{lem25}
$\psd(A+B)\leqslant \psd(A)+\psd(B)$.
\end{lem}

\begin{lem}\label{lem24}
$\psd(A\circ A)\leqslant \operatorname{rank}(A)$.
\end{lem}

\begin{lem}\label{lemblock}
$\psd\left(\begin{array}{c|c}
A&{B}\\\hline
O&C
\end{array}\right)\geqslant\psd(A)+\psd(C).$
\end{lem}

\begin{observation}\label{obs123}
$\psd(A|B)\leqslant\psd(A)+\psd(B)$.
\end{observation}

\begin{observation}\label{obs124}
Let $M$ be a nonnegative matrix, assume that $(A_i)$ and $(B_j)$ are families of $k\times k$ PSD matrices such that $\tr(A_i B_j)=M(i|j)$ for all $i,j$. Denote by $d$ the dimension of the linear space spanned by the rows of all matrices in $(A_\sigma)$. Then $\psd(M)\leqslant d$.
\end{observation}

\begin{proof}
The Gram--Schmidt process allows us to construct an orthonormal basis $e_1,\ldots,e_k$ such that the vectors $e_1,\ldots,e_{k-d}$ belong to the intersection of kernels of all of the $A_\sigma$'s. We can assume without loss of generality that the tuples $(A_i)$ and $(B_j)$ are written with respect to this basis, and then the first $k-d$ rows of $A_\sigma$ are zero. Therefore, all the non-zero elements of the $A_\sigma$'s are located in the bottom-right $d\times d$ submatrices. So we get $\tr(A_i B_j)=\tr(A'_iB'_j)$, where $A'_i$, $B'_j$ are matrices obtained from $A$, $B$ by cutting off the first $k-d$ rows and columns.
\end{proof}

Now we compute the PSD rank of a relevant specific matrix.

\begin{ex}\label{ex22}
The PSD rank of the matrix
$$P(\alpha)=\begin{pmatrix}
\alpha&1&1\\
1&1&0\\
1&0&1
\end{pmatrix}$$
equals two whenever $\alpha\in[0,4]$.
\end{ex}

\begin{proof}
We get $\psd(P)\geqslant2$ by Lemma~\ref{ex21}. To prove a reversed inequality, we set $A_2=B_2=\left(\begin{smallmatrix}
1 & 0\\
0 & 0
\end{smallmatrix}\right)$, $A_3=B_3=\left(\begin{smallmatrix}
0 & 0\\
0 & 1
\end{smallmatrix}\right)$, $A_1=\left(\begin{smallmatrix}
1 & a\\
a & 1
\end{smallmatrix}\right)$, $B_1=\left(\begin{smallmatrix}
1 & b\\
b & 1
\end{smallmatrix}\right)$, where $a,b\in[-1,1]$ are such that $2ab+2=\alpha$. The $A_i$'s and $B_j$'s are PSD, and we can check that $\tr(A_iB_j)=P(i|j)$ for all $i,j$.
\end{proof}

The following is a key result of this section.

\begin{lem}\label{pr3}
Let $S\in\mathbb{R}^{n\times n}$, $b\in\mathbb{R}^{n\times 1}$, $c\in\mathbb{R}^{1\times n}$ be nonnegative matrices, and let $N$ be a positive integer. Assume that the PSD rank of the matrix
$$
G=\left(\begin{array}{ccc|c|cccc}
&&&&0&0\\
&{S}&&{b}&\vdots&\vdots\\
&&&&0&0\\\hline
&{c}&&{N}&{N}&{N}\\\hline
0&\ldots&0&{N}&{N}&{0}\\
0&\ldots&0&{N}&{0}&{N}
\end{array}\right)
$$
does not exceed $r+2$. Then there is a nonnegative number $x$ such that the matrix
$\mathcal{A}(x)=\left(\begin{array}{c|c}
S&{b}\\\hline
c&{x}
\end{array}\right)$
has PSD rank at most $r$.
\end{lem}

\begin{proof} We label the rows and columns of $G$ by $1,\ldots,n,\nu_1,\nu_2$ with respect to the order they appear in the definition of $G$. There are two tuples, $(A_\sigma)$ and $(B_\tau)$, of PSD matrices of order $r+2$ such that $\tr(A_\sigma B_\tau)=G(\sigma|\tau)$ for all $\sigma,\tau$. Using the Cholesky decomposition, we write
\begin{equation}\label{eq119}A_\sigma=\sum_{i} u_{\sigma i}u_{\sigma i}^\top,\,\,\,
B_\tau=\sum_{j} v_{\tau j}v_{\tau j}^\top,\end{equation}
where $i,j$ enumerate the summands in the decomposition. Here and in the rest of the proof, we assume that they run over $\{1,\ldots,r+2\}$. We get from~\eqref{eq119} that
\begin{equation}\label{eq120}\tr(A_\sigma B_\tau)=\sum_{i,j}\left(u_{\sigma i}^\top v_{\tau j}\right)^2.\end{equation}
Let us denote by $U$ the linear space spanned by the vectors in the decompositions of $A_{\nu_1}$ and $A_{\nu_2}$, that is, by the set $\cup_i\{u_{\nu_1i},u_{\nu_2i}\}$. Similarly, we denote by $V$ the linear space spanned by the vectors in the decompositions of $B_{\nu_1}$ and $B_{\nu_2}$. Example~\ref{ex21} and Observation~\ref{obs124} show that $\dim U\geqslant2$, $\dim V\geqslant2$, the equality~\eqref{eq120} shows that the vectors $u_{\sigma i}$ belong to the orthogonal complement $V^\bot$ if $\sigma\in\{1,\ldots,n-1\}$.

\textit{Step 1.} Since $\dim U+\dim U^\bot=r+2$, one has $\dim U^\bot\leqslant r$. Similarly, we have $\dim V^\bot\leqslant r$.

\textit{Step 2.} If $\dim V^\bot\leqslant r-1$, then the matrix obtained from $\mathcal{A}$ by removing the last row has PSD rank at most $r-1$ by Observation~\ref{obs124}, and we get from Observation~\ref{obs123} that $\psd\mathcal{A}(x)\leqslant r$ for all $x$. This would complete the proof of the lemma, so we can assume that $\dim V^\bot\geqslant r$, or, taking into account Step~1, that $\dim V^\bot=r$. Similarly, we get $\dim U^\bot=r$. Using Step~1 again, we get $\dim U=\dim V=2$.

\textit{Step 3.} Now assume that $\dim (U+V^\bot)<r+2$. In this case, Observation~\ref{obs124} shows that the matrix obtained from $G$ by removing the $n$th row has PSD rank at most $r+1$, so we get from Lemma~\ref{lemblock} that the matrix obtained from $\mathcal{A}$ by removing the last row has PSD rank at most $r-1$. As in Step~2, we conclude that $\psd\mathcal{A}(x)\leqslant r$ for all $x$, which would complete the proof of the lemma. So we can assume that $\dim (U+V^\bot)\geqslant r+2$, which implies that $\mathbb{R}^{r+2}$ is a direct sum of $U$ and $V^\bot$ after taking into account the result of Step~2. Similarly, we get that $\mathbb{R}^{r+2}$ is a direct sum of $V$ and $U^\bot$.

\textit{Step~4.} By the result of Step~3, we have $u_{ni}=\alpha_i+\beta_i$ with $\alpha_i\in U$, $\beta_i\in V^\bot$, and $v_{ni}=\alpha'_i+\beta'_i$ with $\alpha'_i\in V$ and $\beta'_i\in U^\bot$. We define $A'_n=\sum_i\beta_{n i}\beta_{n i}^{\top}$, $B'_n=\sum_i\beta_{n i}\beta_{n i}^{\prime\top}$, we replace the matrices $A_n$, $B_n$ by $A'_n$, $B'_n$ in the tuples $(A_\sigma)$, $(B_\tau)$, and we denote the newly obtained tuples by $(A'_\sigma)$, $(B'_\tau)$.

We note that, for all $p,q\in\{1,\ldots,n\}$, one has $\tr(A'_pB'_q)=\tr(A_pB_q)$ unless $p=q=n$. This means that the tuples $(A'_p)$, $(B'_q)$ provide a valid factorization for the matrix $\mathcal{A}(y)$ with some value of $y$. Since for all $p\in\{1,\ldots,n\}$ the rows of the matrices $A'_p$ belong to $V^\bot$, we apply Observation~\ref{obs124} and get $\psd(\mathcal{A}(y))\leqslant\dim V^\bot$. So we get $\psd(\mathcal{A}(y))\leqslant r$ by the result of Step~2.
\end{proof}

\section{A matrix completion problem}

Let us consider the set $\R\cup\{\ca\}$, where the element $\ca$ can be thought of as a `real number that is not yet specified'. A matrix $S$ with entries in $\R\cup\{\ca\}$ is called \textit{incomplete}, and any real matrix $C$ for which $C(i|j)=S(i|j)$ whenever $S(i|j)\in\R$ is called a \textit{completion} of $S$. It will be important for us to consider the following technical condition imposed on incomplete matrices.

\begin{defn}\label{defsqrt}
Consider the following property of an incomplete matrix $S$:

For any column index $k$, there are row indexes $i_1,i_2$ and column indexes $j_1,j_2$ such that the submatrix $S(i_1,i_2|k,j_1,j_2)$ equals $\left(\begin{smallmatrix}
 0& 1 & 0\\
 0& 0 & 1
\end{smallmatrix}\right)$.

If this holds for both $S$ and $S^\top$, we say that $S$ satisfies the \textit{sqrt condition}.
\end{defn}

This condition is named after the matrix invariant known as \textit{sqrt rank}, see~\cite{PSD} for details. A reader familiar with this notion can give an equivalent formulation of the lemma below: For any completion $C$ of a matrix satisfying sqrt condition, $\psd(C)\leqslant3$ implies $\operatorname{sqrt\,rank}(C)\leqslant3$.

\begin{lem}\label{lem34}
Let $S$ be an incomplete matrix satisfying the sqrt property. If a completion $C$ of $S$ satisfies $\psd(C)\leqslant3$, then there is a real matrix $Q$ such that $Q\circ Q=C$ and $\operatorname{rank}(Q)\leqslant3$.
\end{lem}

\begin{proof}
Let $A_1,\ldots,A_n, B_1,\ldots,B_m$ be $3\times3$ PSD matrices satisfying $\tr(A_iB_j)=C(i|j)$ for all $i,j$. If $\rk(A_i)\leqslant1$ and $\rk(B_j)\leqslant1$, then we can write $A_i=a_ia_i^\top$ and $B_j=b_jb_j^\top$; in this case, the matrix $Q$ defined as $Q(i|j)=a_i^\top b_j$ satisfies the assumptions of the lemma.

So we can assume that either $\rk(A_i)>1$ for some $i$ or $\rk(B_k)>1$ for some $k$. Since the sqrt condition is invariant under transposition, we assume that $\rk(B_k)>1$ holds. We get $B_k=u_1u_1^\top+u_2u_2^\top+u_3u_3^\top$ with non-collinear $u_1,u_2$ (and possibly zero $u_3$). Let $v$ be a non-zero vector orthogonal to $u_1,u_2$; we get $A_t=\alpha_t vv^\top$ for all $t$ such that $C_{tk}=0$. In other words, any two rows of $C$ having zeros at $k$th positions are collinear, which contradicts the sqrt condition.
\end{proof}


Now we need to introduce the matrix which is the key of our redution from ETR to PSD RANK. This is done in the separate definition for the ease of subsequent references.

\begin{defn}\label{defnM}
Let $K$ be a positive integer and $S$ an $n\times n$ incomplete matrix. We are going to assign an instance of Problem~\ref{prob11} to every pair $(S,K)$ as follows. We denote by $E(S)=\{e_1=(i_1j_1),\ldots,e_k=(i_kj_k)\}$ the set of all pairs $(i,j)$ such that $S(i|j)=\ca$. The matrix $\M=\M(S,K)$ will have $2k+n$ rows and columns indexed with $E^1\cup E^2\cup\nov$, where $E^1, E^2$ are copies of the set $E(S)$ and $\nov=\{1,\ldots,n\}$. We define $\M(i|j)=S(i|j)$ if $(i,j)\notin E$; if $e=(i,j)\in E$, then we define the submatrix $\M(i,e^1,e^2|j,e^1,e^2)$ as $KP(1)$, where $P(\alpha)$ is as in Example~\ref{ex22}. The entries of $\M$ that are not yet defined are set to equal $0$.
\end{defn}

\begin{lem}\label{lem41}
If $S$ admits a completion $C$ such that $\psd(C)\leqslant3$ and $\max_{i,j}|C(i|j)|\leqslant K$, then $\psd(\M(S,K))\leqslant2k+3$.
\end{lem}

\begin{proof}
By the definition of $\M$, the nonzero entries of $\M-C$ belong to $k$ disjoint submatrices of the form $KP(\alpha)$ with $\alpha\in[0,1]$. The PSD ranks of these submatrices equal $2$ by Example~\ref{ex22}, so the result follows from Lemma~\ref{lem25}.
\end{proof}

\begin{lem}\label{lem42}
Let $S$ be an incomplete matrix. If $\psd(\M(S,K))\leqslant2k+3$, then $S$ admits a completion $C$ such that $\psd(C)\leqslant3$. 
\end{lem}

\begin{proof}
There is nothing to prove if $E$ is empty, so we assume $e=(i,j)\in E$ and proceed by the induction on $k$ (that is, on the cardinality of $E$). By Lemma~\ref{pr3}, there is a positive number $x_{ij}$ such that, if we remove from $\M$ the rows and columns with indexes $e^1, e^2$ and replace the $(i,j)$ entry with $x_{ij}$, we will get the matrix $\M'$ with PSD rank at most $2k+1$. Note that we have $\M'=\M(S',K)$, where $S'$ is the matrix obtained from $S$ by replacing the $(i,j)$ entry with $x_{ij}$. Since the number of $\ca$-entries of $S'$ is less than that in $S$, the result follows by induction.
\end{proof}

\section{Matrix completion problems and ETR}

Now we will work with the $\exists\R$-complete problem discussed in Theorem~\ref{lem14}. A polynomial $f(x_1,\ldots,x_n)$ as in this theorem can be written as a sum of monomials. 
If $p=\pm x_{i_1}x_{i_2}\ldots x_{i_k}$ is such a monomial, then we define the set $\sigma(p)$ as $\{\pm1,\pm x_{i_1},\pm x_{i_1}x_{i_2},\ldots,\pm p\}$. If $f=p_1+\ldots+p_s$, then we define $$\sigma(f)=\sigma(p_1)\cup\ldots\cup\sigma(p_s)\cup\{0,\pm p_1,\pm(p_1+p_2),\ldots,\pm P\}.$$ Clearly, the construction of the set $\sigma=\sigma(f)$ can be done in time polynomial in the length of $f$.

We denote by $\H=\H(f)$ the set of those vectors in $\sigma^3$ that have one of the coordinates equal to $1$. We proceed with the definition of the three matrices, $\A=\A(f)$, $\B=\B(f)$, $\C=\C(f)$, which have their rows and columns indexed with the elements of $\H$.

\begin{defn}\label{defna}
Define $\A_1$ as the matrix whose rows are vectors in $\H$, and let $\A_2$ be the matrix whose columns are vectors in $\H$. We define the matrix $\A$ over $\mathbb{Z}[x_1,\ldots,x_n]$ as $\A_1\A_2$, that is, we set $\A(u|v)=(u\cdot v)^2$, where $u\cdot v$ is the dot product of $u$ and $v$ as vectors in $\sigma^3$. (In terms of conventional matrix multiplication, $u\cdot v$ stands for $u^\top v$.)
\end{defn}

\begin{defn}\label{defnb}
$\B$ is the matrix with entries in $\mathbb{Z}\cup\{\ca\}$, and we set $\B(u|v)=b$ if $\A(u|v)$ is identically equal to $b$, we also set $\B(u|v)=0$ in the case when $\A(u|v)$ is a multiple of $f$. In the remaining cases, we set $\B(u|v)=\ca$.
\end{defn}

\begin{defn}\label{defnc}
$\C$ is the matrix with entries in $\{0,\ast,\ca\}$. We set $\C(u|v)=0$ if $\B(u|v)=0$, we set $\C(u|v)=\ca$ if $\B(u|v)=\ca$, and also $\C(u|v)=\ast$ if $\B(u|v)$ is a nonzero number.
\end{defn}

As in the above discussion, the symbol $\ca$ can be thought of as a number that is not yet defined. In particular, $\B$ belongs to the class of \textit{incomplete matrices} discussed above. The new symbol, $\ast$, has a similar meaning as $\ca$ but denotes a number required to be non-zero. One can check that these matrices can be constructed in polynomial time. 



\begin{observation}\label{obs71}
$\B$ satisfies the sqrt condition as in Definition~\ref{defsqrt}.
\end{observation}

\begin{proof}
Let $j\in\sigma^3$ be a column index; by the definition, one of the coordinates of $j$ equals $1$. We will consider the case when the first coordinate of $j$ equals one, and all the other cases can be considered symmetrically. So we have $j=(1,f,g)$; we define $i_1=(-f,1,0)$, $i_2=(-g,0,1)$, $j_1=(0,1,0)$, $j_2=(0,0,1)$, and we get $\B(i_1,i_2|j,j_1,j_2)=\left(\begin{smallmatrix}
 0& 1 & 0\\
 0& 0 & 1
\end{smallmatrix}\right)$. Since $\B=\B^\top$, the result follows.
\end{proof}

\begin{lem}\label{obs72}
If $f(\xi_1,\ldots,\xi_n)=0$ for some $\xi_i\in\R$, then there is a completion $B'$ of $\B$ such that $\psd(B')\leqslant3$ and $|B'(u|v)|\leqslant9\left(\operatorname{length} f\right)^4$.
\end{lem}

\begin{proof}
We define $B'$ as the entrywise square of the matrix $A'=\A(\xi_1,\ldots,\xi_n)$ obtained from $\A$ by substituting the variables $x_i$ with numbers $\xi_i$. By the definition, the matrix $A'$ is a product of the $|\H|\times 3$ matrix $\A_1(\xi_1,\ldots,\xi_n)$ and the $3\times |\H|$ matrix $\A_2(\xi_1,\ldots,\xi_n)$. We get $\operatorname{rank}(A')\leqslant3$, and Lemma~\ref{lem24} implies $\psd(B')\leqslant3$.

Recall that $f$ satisfies the assumption of Theorem~\ref{lem14}, so that $\xi_i\in[-1,1]$. In particular, the value of no monomial in the $\xi_i$'s can exceed one. We see that the absolute values of the entries of $\A_1(\xi_1,\ldots,\xi_n)$ and $\A_2(\xi_1,\ldots,\xi_n)$ cannot exceed $\left(\operatorname{length} f\right)$, which implies the desired bound on the entries of $B'$.
\end{proof}

Now we are going to prove Observation~\ref{obs72} in the converse direction. That is, we want to show that existence of a completion of $\B$ with PSD rank three implies that $f=0$ is satisfiable. As we will see later, the following is a stronger statement. We say that a matrix $D$ is a completion of the matrix $\C$ as in Definition~\ref{defnc} if $\C(u|v)=0$ implies $D(u|v)=0$ and $\C(u|v)=\ast$ implies $D(u|v)\neq0$.

\begin{lem}\label{lem73}
If the matrix $\C(f)$ admits a completion $D$ such that $\operatorname{rank}(D)\leqslant 3$, then the formula $f=0$ is satisfiable.
\end{lem}

\begin{proof}
\textit{Step 1.} The completion $D$ satisfying $\operatorname{rank}(D)\leqslant 3$ can be written as the product of an $|\H|\times3$ matrix and a $3\times|\H|$ matrix. In other words, there are families $(p_u)$, $(l_v)$ of vectors in $\mathbb{R}^3$ such that \begin{equation}\label{eqpr1}p_u\cdot l_v\neq0 \mbox{ $ $ if $ $}\C(u|v)=* \mbox{ $ $ and $ $} p_u\cdot l_v=0 \mbox{ $ $ if $ $} \C(u|v)=0.\end{equation} (The indexes $u,v$ run over the set $\H$.) Let us make the two easy observations, denoted Step~2 and Step~3 for the ease of reference.

\textit{Step 2.} The properties~\eqref{eqpr1} will not be invalidated under the transformation $(p_u, l_v)\to(B^\top p_u, B^{-1} l_v)$, where $B$ is a non-singular $3\times 3$ matrix.

\textit{Step 3.} The properties~\eqref{eqpr1} will not be invalidated under the transformation $(p_u, l_v)\to(\lambda_u p_u, \mu_v l_v)$, where $\lambda_u,\mu_v$ are scalars.

\textit{Step 4.} Using Step~2, we can assume without loss of generality that $p_{(1,0,0)}=(1,0,0)$, $p_{(0,1,0)}=(0,1,0)$, $p_{(0,0,1)}=(0,0,1)$, and then the conditions $p_{(1,0,0)}\cdot l_{(1,0,0)}\neq0$, $p_{(0,1,0)}\cdot l_{(1,0,0)}=0$, $p_{(0,0,1)}\cdot l_{(1,0,0)}=0$ imply that $l_{(1,0,0)}=(\lambda_1,0,0)$ with $\lambda_1\neq0$. Similarly, we get $l_{(0,1,0)}=(0,\lambda_2,0)$, $l_{(0,0,1)}=(0,0,\lambda_3)$ with non-zero $\lambda_2,\lambda_3$.

\textit{Step 5.} Further, we get from the conditions $p_{(1,1,1)}\cdot l_{(1,0,0)}\neq0$, $p_{(1,1,1)}\cdot l_{(0,1,0)}\neq0$, $p_{(1,1,1)}\cdot l_{(0,0,1)}\neq0$ that $p_{(1,1,1)}=(a,b,c)$ with $0\notin\{a,b,c\}$. Applying Step~2 again, we can assume that $p_{(1,1,1)}=(1,1,1)$.

\textit{Step 6.} Applying the transformations as in Steps~2 and~3 to vectors in Steps~4 and~5, we can assume without loss of generality that $p_{(1,0,0)}=l_{(1,0,0)}=(1,0,0)$, $p_{(0,1,0)}=l_{(0,1,0)}=(0,1,0)$, $p_{(0,0,1)}=l_{(0,0,1)}=(0,0,1)$, $p_{(1,1,1)}=(1,1,1)$.

\textit{Step 7.} The conditions $p_{(1,0,0)}\cdot l_{(1,0,x_i)}\neq0$ and $p_{(0,1,0)}\cdot l_{(1,0,x_i)}=0$ show that $l_{(1,0,x_i)}=(a,0,d)$ with non-zero $a$. Using Step~3, we can assume that $l_{(1,0,x_i)}=(1,0,y_i)$ for any variable $x_i$. In what follows, $y_i$ denotes the third coordinate of the vector $l_{(1,0,x_i)}$, and we write $x=(x_1,\ldots,x_n)$, $y=(y_1,\ldots,y_n)$.

\textit{Step~8.} The strategy of the rest of the proof is to check that $f(y)=0$. The use of Step~3 allows us to think of the vectors $p_u$, $l_v$ as those defined up to a scalar multiplication. We say that the label $u=(a(x),b(x),c(x))$ is \textit{row-good} (or \textit{column-good}) if the vector $p_u$ (or $l_u$, respectively) is collinear to $(a(y),b(y),c(y))$.

\textit{Step~9.} The considerations of Steps~6,~7 show that the vectors $(1,0,0)$, $(0,1,0)$, $(0,0,1)$, $(1,1,1)$ are row-good, and the vectors $(1,0,0)$, $(0,1,0)$, $(0,0,1)$, $(1,0,x_i)$ are column-good. Similarly, $l_{(0,-1,1)}$ is orthogonal to the row-good vectors $(1,0,0)$ and $(1,1,1)$, so that $(0,-1,1)$ is column-good. 

\textit{Step~10.} Now assume that a vector $(g,0,h)$ is column-good. The vector $p_{(-h,g,g)}$ is then orthogonal to column-good vectors $(g,0,h)$ and $(0,-1,1)$, so that $(-h,g,g)$ is a row-good vector. Now the vector $(g,h,0)$ is column-good because $l_{(g,h,0)}$ is orthogonal to row-good vectors $(0,0,1)$ and $(-h,g,g)$.

\textit{Step~11.} If $(g,0,h)$ is column-good, the vector $(-h,0,g)$ is row-good because it is orthogonal to column-good vectors $(g,0,h)$ and $(0,1,0)$. The symmetry and Step~10 imply that, in the case when $(g,0,h)$ is column-good, any permutation of $(g,h,0)$ is column-good and any permutation of $(-h,0,g)$ is row-good.

\textit{Step~12.} Let us now assume that $(1,0,\alpha)$, $(1,0,\beta)$ are column-good.

\textit{Step~12.1.} Assume $\alpha+\beta\in\sigma$. We see that $(-1,1,\alpha)$ is column-good because it is orthogonal to vectors $(1,1,0)$ and $(0,-\alpha,1)$, which are row-good by Steps~9 and~11. Now we see that $(-\beta,-\alpha-\beta,1)$ is row-good because it is orthogonal to the column-good vectors $(-1,1,\alpha)$ and $(1,0,\beta)$. Finally, the vector $(0,1,\alpha+\beta)$ is column-good because it is orthogonal to the row-good vectors $(-\beta,-\alpha-\beta,1)$ and $(1,0,0)$. The vector $(1,0,\alpha+\beta)$ is column-good by Step~11.

\textit{Step~12.2.} Assume $\alpha-\beta\in\sigma$. We see that $(-1,-1,\alpha)$ is row-good because it is orthogonal to the vectors $(1,-1,0)$ and $(0,\alpha,1)$, which are column-good by Steps~9 and~11. Now we see that $(\beta,\alpha-\beta,1)$ is column-good because it is orthogonal to the row-good vectors $(-1,-1,\alpha)$ and $(1,0,-\beta)$. The vector $(0,1,\beta-\alpha)$ is row-good because it is orthogonal to the column-good vectors $(\beta,\alpha-\beta,1)$ and $(1,0,0)$. Finally, we see that $(0,\alpha-\beta,1)$ is a column-good vector since it is orthogonal to the row-good vectors $(1,0,\beta-\alpha)$ and $(0,1,0)$. Step~11 shows that $(1,0,\alpha-\beta)$ is column-good as well.

\textit{Step~12.3.} Assume $\alpha\beta\in\sigma$. The vector $(\alpha\beta,1,\alpha)$ is column-good because it is orthogonal to the vectors $(0,-\alpha,1)$ and $(1,0,-\beta)$, which are row-good by Step~10. The vector $(1,-\alpha\beta,0)$ is row-good because it is orthogonal to the column-good vectors $(0,0,1)$ and $(\alpha\beta,1,\alpha)$. Finally, we see that $(\alpha\beta,1,0)$ is a column-good vector because it is orthogonal to row-good vectors $(1,-\alpha\beta,0)$ and $(0,0,1)$. Step~10 implies that $(1,0,\alpha\beta)$ is column-good as well.

\textit{Step~13.} The results of Step~12 show that the vector $(1,0,s)$ is column-good for all $s\in\sigma$. We get that $p_{(0,0,1)}\cdot l_{(1,0,f)}=0$ because $\C((0,0,1)|(1,0,f))=0$, which means that $(0,0,1)\cdot(1,0,f(y))=0$ or $f(y)=0$. The proof is complete. 
\end{proof}

We conclude the section with the desired proof of the statement converse to Observation~\ref{obs72}.

\begin{lem}\label{lem74}
If the matrix $\B(f)$ admits a completion $B'$ such that $\psd(B')\leqslant 3$, then the formula $f=0$ is satisfiable.
\end{lem}

\begin{proof}
Observation~\ref{obs71} shows that $\B$ possesses the sqrt property, so Lemma~\ref{lem34} is applicable. Since $\B$ admits a completion $B'$ with $\psd(B')\leqslant 3$, we see that there is a matrix $C'$ with rank at most three such that $B'=C'\circ C'$. We see that $C'$ is a rank-three completion of $\C$, so the formula $f=0$ is satisfiable by Lemma~\ref{lem73}.
\end{proof}

\section{The main result}

Let us finalize the proof of the main result. Theorem~\ref{thr11} follows directly from Remark~\ref{remr11} and the theorem below.

\begin{thr}\label{thrm}
There is a polynomial-time reduction from ETR to PSD rank.
\end{thr}

\begin{proof}
The problem in Theorem~\ref{lem14} is $\exists\R$-complete, so we can construct a reduction from it. For an instance $f$ of this problem, we construct the matrix $\B=\B(f)$ as in Definition~\ref{defnb}, we denote by $k$ the number of $\ca$-entries of $\B$, and we set $K=9\left(\operatorname{length} f\right)^4$. We construct the matrix $\M=\M(\B,K)$ as in Section~6, and we claim that
$\rho: f\to\left(\M,2k+3\right)$ is a desired reduction. As we noted in the above discussion, the matrix $\B$ can be constructed in time polynomial in the length of $f$; the matrix $\M$ can be constructed in time polynomial in the length of $(\B,K)$, so the function $\rho$ can be computed in polynomial time.

Now let $f$ be a yes-instance. By Lemma~\ref{obs72}, there is a completion $B'$ of $\B$ such that $\psd(B')\leqslant3$ and $|B'(u|v)|\leqslant K$. Lemma~\ref{lem41} implies $\psd(\M)\leqslant2k+3$, that is, $\rho(f)$ is a yes-instance of PSD RANK.

Finally, let $f$ be a no-instance. By Lemma~\ref{lem74}, $\B$ admits no completion $B'$ such that $\psd(B')\leqslant3$. Lemma~\ref{lem42} implies $\psd(\M)>2k+3$, so that $\rho(f)$ is a no-instance of PSD RANK.
\end{proof}

The technique developed in this paper allows us to prove another interesting result. For any subfield $F\subset\R$, we can define a similar PSD rank function but with a requirement that the entries of matrices in the factorization belong to $F$. We consider a polynomial $f\in\mathbb{Z}[x]$ that is irreducible over $\mathbb{Q}$ but has a root $r$ in $\R$. In the notation of the proof of Theorem~\ref{thrm}, the matrix $\M(\B(f),K)$ has PSD rank $2k+3$ with respect to $\mathbb{Q}(r)$. However, the argument as in Lemmas~\ref{lem73} and~\ref{lem74} shows that every rational completion of $\B(f)$ has PSD rank at least four. We can argue as in Lemmas~\ref{pr3} and~\ref{lem42} and prove that the rational PSD rank of $\M$ is greater than $2k+3$. In other words, the PSD ranks defined with respect to $\mathbb{Q}$ and $\mathbb{Q}(r)$ represent different functions of rational matrices, for any irrational number $r\in\R$. This is a generalization of the result proved in~\cite{FGR} by Fawzi, Gouveia, and Robinson.

\end{document}